\documentclass[a4paper,12pt,reqno]{amsart}

\usepackage{amscd,amssymb,graphicx,color,epigraph}

\usepackage[shortalphabetic,initials]{amsrefs}

\usepackage{amssymb}
\usepackage{amsthm}
\usepackage{bm}
\usepackage{latexsym}
\usepackage{amsmath}
\usepackage{eufrak}
\usepackage{mathrsfs}
\usepackage{amscd}
\usepackage[all,cmtip]{xy}
\usepackage{color}
\usepackage{amscd}
\usepackage{listings}
\usepackage{marginnote}

\theoremstyle{definition}

 \numberwithin{equation}{section}

\newtheorem{theorem}{Theorem}[section]
\newtheorem{proposition}[theorem]{Proposition}

\newtheorem{lemma}[theorem]{Lemma}

\theoremstyle{definition}

\newcommand{\appsection}[1]{\let\oldthesection\thesection
\renewcommand{\thesection}{Appendix \oldthesection}
\section{#1}\let\thesection\oldthesection}

\theoremstyle{remark}

\newtheorem{remark}[theorem]{Remark}
\newtheorem{example}[theorem]{Example}

\DeclareMathOperator{\Cl}{Cl}

\DeclareMathOperator{\Pic}{Pic}
\DeclareMathOperator{\CM}{CM}

\def\D{{\mathbb{D}}}

\def\Z{{\mathbb{Z}}}

\def\Q{{\mathbb{Q}}}
\def\C{{\mathbb{C}}}
\def\P{{\mathbb{P}}}

\def\H{{\mathbb{H}}}

\def\cR{{\mathcal{R}}}

\def\O{{\mathcal{O}}}

\def\X{{\mathcal{X}}}

\def\Z{{\mathbb{Z}}}

\def\epsilon{\varepsilon}

\begin{document}
\bibliographystyle{amsplain}

\title{The Coble-Mukai lattice from $\Q$-Gorenstein deformations}

\author[Giancarlo Urz\'ua]{Giancarlo Urz\'ua}
\email{urzua@mat.uc.cl}
\address{Facultad de Matem\'aticas, Pontificia Universidad Cat\'olica de Chile, Campus San Joaqu\'in, Avenida Vicu\~na Mackenna 4860, Santiago, Chile.}

\begin{abstract}
We show some geometric properties of Enriques surfaces via $\Q$-Gorenstein smoothings of Coble surfaces. In particular, we explicitly identify the Enriques lattice of the general fiber with the Coble-Mukai lattice. At the end, we discuss applications to Gorenstein $\Q$-Homology projective planes with trivial canonical class.
\end{abstract}

\date{\today}

\maketitle

\tableofcontents

\section{Introduction} \label{s0}

Our ground field is $\C$. An \textit{Enriques surface} is a nonsingular projective surface $Z$ with $q(Z)=h^1(\O_Z)=0$, $p_g(Z)=h^2(\O_Z)=0$, and $2 K_Z \sim 0$. (An excellent survey on Enriques surfaces is \cite{Do16}, our main reference will be the two volumes book \cite{Enr20I, Enr20II}.) Degenerations of Enriques surfaces have been studied in e.g. \cite{Ku77}, \cite{M81}, \cite{P77}, \cite{K92}, \cite{CM94}. We are interested in $\Q$-Gorenstein degenerations (called moderate degenerations in \cite{K92}). 

An \textit{Enriques W-surface} is a normal projective surface $X$ together with a proper deformation $(X \subset \X) \to (0 \in \D)$, where $\D$ is a smooth curve germ, such that
\begin{enumerate}
\item $X$ has at most \textit{Wahl singularities}, this is, cyclic quotient singularities of type $\frac{1}{n^2}(1,na-1)$ with gcd$(a,n)=1$,
\item $\X$ is a normal $3$-fold with $K_{\X}$ $\Q$-Cartier,
\item the fiber $X_0$ is reduced and isomorphic to $X$,
\item the fiber $X_t$ is an Enrique surface for $t\neq 0$.
\end{enumerate}

(For more general W-surfaces and their birational geometry see \cite{U16a}.) Kawamata proves that the monodromy for this type of degenerations is trivial \cite[Section 2]{K92}, and that an Enriques W-surface $X$ with $K_X$ nef can have only singularities of type $\frac{1}{4}(1,1)$ \cite[Theorem 4.1]{K92} (i.e. these are flower pot degenerations \cite{P77}). In addition, if $X$ is indeed singular, then it must be rational. Let $\phi \colon V \to X$ be the minimal resolution, and let $\{ C_1,\ldots, C_s \}$ be the disjoint exceptional $(-4)$-curves. Then $$-2K_V \sim C_1 + \ldots + C_s,$$ and so $V$ is a \textit{Coble surface of K3 type} \cite[Section 9.1]{Enr20II}. By \cite[Corollary 9.1.5]{Enr20II}, we have that $s \leq 10$.  

For any given Enriques W-surface $X$, the purpose of this note is to directly compute some geometric properties of the Enriques surfaces $X_t$. For example, this point of view gives a direct way to identify the Coble-Mukai lattice of $V$ with the Picard group of $X_t$ modulo canonical class (see Theorem \ref{cm=pic}). The Coble-Mukai lattice was introduced by Mukai with no reference to a degeneration of Enriques surfaces. In \cite[Theorem 9.2.15]{Enr20II} it is proved that the Coble-Mukai lattice is isomorphic to the Enriques lattice, but there are no degenerations of Enriques surfaces involved either. We also explicitly describe $\Pic(X_t)$ from classes of curves in the singular surface $X$. As an application, we discuss nodal Enriques surfaces from Coble surfaces, and in particular some open questions on $\Q$-homology projective planes via Enriques W-surfaces.

\subsubsection*{Acknowledgments}
The author would like to thank Igor Dolgachev for the motivation to write this note, and for various useful comments and references. The author also thanks Matthias Sch\"utt for his interest and useful comments. The author is grateful to Phillip Griffiths and Carlos Simpson for the invitation to talk at the ``Moduli and Hodge theory" online conference at the Institute of the Mathematical Sciences of the Americas (IMSA) during February 1-5 2021, and for the invitation to write a paper for the corresponding proceedings. The author would like to thank the anonymous referee for interesting and useful observations. The author was supported by the FONDECYT regular grant 1230065. 

\section{Preliminaries} \label{s1}

Let $X$ be an Enriques W-surface with $K_X$ nef. Let $$\phi \colon V \to X$$ be the minimal resolution of $X$. As it was said in the introduction, the surface $V$ must be a \textit{Coble surface of K3 type}. In this way, we have two options for $V$ (see \cite[Prop. 9.1.3]{Enr20II}): it is the blow-up of either:

\begin{itemize}
\item[1.] an Halphen surface of index two \footnote{An Halphen surface of index two is a rational elliptic surface with a multiplicity two fiber, and no $(-1)$-curves in the fibers.} over the singularities of one reduced fiber, which is of type $II$, $III$, $IV$, or $I_n$, of , or

\item[2.] a Jacobian rational minimal elliptic fibration over the singularities of two reduced fibers of type $II$, $III$, $IV$, or $I_n$.
\end{itemize}

The blow-ups are as follows. Over a type $II$ fiber we only blow-up the cusp to obtain one $(-4)$-curve. For type $III$ we blow-up twice over the tangent point, and so we obtain two $(-4)$-curves over the fiber. In the case $IV$ we blow-up $4$ times over the triple point, so that we obtain four $(-4)$-curves. Finally, for any $I_n$ fiber we blow-up the $n$ nodes, so that we obtain $n$ $(-4)$-curves. 

Let $\pi \colon V \to S$ be the composition of the blow-ups. We obtain the boundary curves $\{C_1,\ldots,C_s\}$ in $V$ from the distinguished singular fiber(s). We recall that $C_i \simeq \P^1$ and $C_i^2=-4$, i.e. they are the $(-4)$-curves over the fibers. The contraction of the boundary curves gives the surface $X$. 

We will only consider the case when the singular fibers are of type $I_n$, the other cases are degenerations. Hence for one singular fiber we have $I_s$, and for two we have $I_{s_1}$ and $I_{s_2}$ with $s_1+s_2=s$. We have that $1 \leq s \leq 10$ by \cite[Corollary 9.1.4]{Enr20II}. 

One can actually prove that such an $X$ has no local-to-global obstructions to deform, and that any $\Q$-Gorenstein smoothing of it is an Enriques surface (see \cite[Theorem 4.2(0)]{U16b}, where there is a result for more general elliptic fibrations). The $\Q$-Gorenstein deformation space of $X$ has dimension $10$ (see \cite[Section 3]{H16}).   

\begin{example} [Blow-ups of Enriques W-surfaces] As above, consider an elliptic fibration $S \to \P^1$ with sections and at least two $I_1$ fibers $F_1$ and $F_2$. Let $V \to S$ be the blow-up of both nodes. We now blow-up over a general point in, say, $F_2$ to obtain the chain $[6,2,2]$ with a $(-1)$-curve intersecting the middle $(-2)$-curve, where the $(-6)$-curve is the proper transform of $F_2$ (see Figure \ref{f0}). We contract $[4]$ and $[6,2,2]$ to obtain a surface $Z$. One can prove that $Z$ has no local-to-global obstructions to deform. We consider a $\Q$-Gorenstein smoothing of $Z$. One can show that the general fiber is the blow-up of an Enriques surface at one point. We note that $K_{Z}$ is not nef because the image $A$ of the $(-1)$-curve intersecting the middle $(-2)$-curve is a negative curve for $K_{Z}$. There is a divisorial contraction of the family which contracts the $(-1)$-curves in the general fibers and $A$ in $Z$, and the resulting surface is an Enriques W-surface. 

\begin{figure}[htbp]
\includegraphics[width=2.5in]{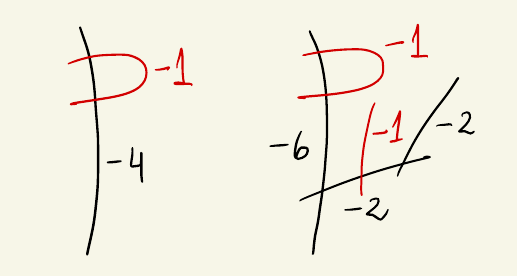}
\caption{The Wahl chains $[4]$ and $[6,2,2]$}
\label{f0}
\end{figure}

These blow-ups of Enriques W-surfaces come from the divisorial contraction universal family \cite{HTU12} defined by $\frac{1}{4}(1,1)$. This is explained in \cite[Section 2]{U16b}, precisely this particular divisorial contraction family is \cite[Example 2.13]{U16b}, where the combinatorial data with all the infinitely many possible Wahl singularities is given by $$[4]-[2,\bar{2},6]-[2,2,2,\bar{2},8]-[2,2,2,2,2,\bar{2},10]-\cdots .$$ For example, in Figure \ref{f3} we have a way to produce a similar situation but with three Wahl singularities, where the $(-1)$-curve in the general fiber (blow-up of an Enriques surface at one point) degenerates into the image of the $(-1)$-curve between $[8,2,2,2,2]$ and $[6,2,2]$.

\begin{figure}[htbp]
\includegraphics[width=2.5in]{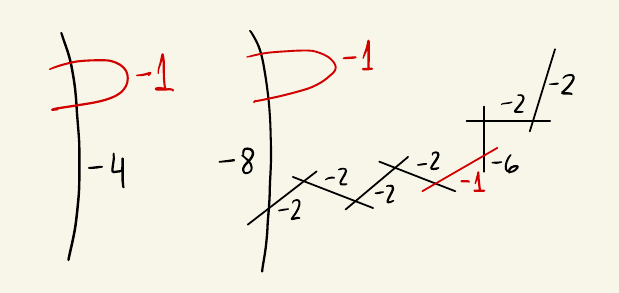}
\caption{The Wahl chains $[4]$, $[6,2,2]$, and $[8,2,2,2,2]$}
\label{f3}
\end{figure}

\end{example}

\begin{example}
In \cite[Example 4.5]{K92}, Kawamata shows the local picture of a $\Q$-Gorenstein smoothing of an elliptic fibration with a singular $I_d$ fiber with Wahl singularities $\frac{1}{n^2}(1,na-1)$ at the nodes. Kawamata explains how these local elliptic surface $\Q$-Gorenstein smoothens into other elliptic surfaces, and the degeneration is explained as the concurrence of a smooth fibre with multiplicity $n$ with $d$ $I_1$ fibers. 
\end{example}

Let $\beta_i$ be the class of $C_i$ in $\Pic(V)$. As in \cite[Chapter 9]{Enr20II}, we define the \textit{Coble-Mukai lattice} of $V$ as $$\CM(V):=\{ x \in \widetilde{\Pic}(V) \colon x \cdot \beta_i=0 \ \text{for all} \ i \}$$ where $\widetilde{\Pic}(V)$ is the lattice in $\Pic(V)_{\Q}$ generated by $\Pic(V)$ and the rational classes $\frac{1}{2} \beta_i$.

\section{Class and Picard groups} \label{s2}

Let $X$ be an Enriques W-surface with $K_X$ nef. We only consider Coble surfaces $V$ of K3 type from a rational minimal elliptic fibration $S \to \P^1$ with sections. Hence we have blow-ups at the nodes of two singular fibers $I_{s_1}$, $I_{s_2}$ with $s_1+s_2=s$. We have the diagram $$ \xymatrix{  & V  \ar[ld]_{\pi} \ar[rd]^{\phi} &  \\ S &  & X}$$ where the morphism $\phi$ is the minimal resolution of $X$, and $\pi$ is a composition of blow ups from $S$. The $(-4)$-curves $\{C_1,\ldots,C_s \}$ are the exceptional curves of $\phi$. We choose in $V$ a chain of smooth rational curves $$C_1 - (-1) - C_2 - (-1) - \ldots - (-1) - C_s $$ using one section from $S \to \P^1$, and $(-1)$-curves from the blow-up $\pi \colon V \to S$. The contraction of that chain defines a surface $W$ with one singularity of type $\frac{1}{4s}(1,2s-1)$, whose Hirzebruch-Jung continued fraction is $[3,2,\ldots,2,3]$ (with $(s-2)$ $2$'s). Hence we have a contraction $\sigma \colon X \to W$, and $W$ has no local-to-global obstructions to deform (same proof as for $X$). 

We now look at $\Q$-Gorenstein deformations of $W$, and then we will do it for $X$. All deformations below happen over $\D$, and $\Q$-Gorenstein deformation means that the canonical class for the corresponding $3$-fold is $\Q$-Cartier.

\begin{proposition}
The surface $W$ can be $\Q$-Gorenstein deformed into a surface $W_t$ with singularities of types either $A_{e_1-1}$, $\ldots$, $A_{e_r-1}$ or
$\frac{1}{4e_1}(1,2e_1-1),A_{e_2-1}$, $\ldots$, $A_{e_r-1}$, where $e_1$, $\ldots$, $e_r$ is a partition of $s$. In the case when we have only singularities of type $A_{e_i-1}$, the minimal resolution of $W_t$ is an Enriques surface.
\label{ade}
\end{proposition}

\begin{proof}
We know that there are no local-to-global obstructions to deform $W$. The rest is \cite[Proposition 2.3]{HP10}.
\end{proof}

The partial resolution $\sigma \colon X \to W$ is locally an M-resolution of the singularity in $W$ (cf. \cite{BC94}). Any $\Q$-Gorenstein deformation of the singularity in $W$ is the blowing-down deformation (see \cite[(11.4)]{KM92}) of a $\Q$-Gorenstein deformation of the exceptional divisor of $\sigma$. The deformations in Proposition \ref{ade} into singular surfaces happen over a proper analytic closed set of the deformation space of the singularity \cite{HP10}, and so for almost all $\Q$-Gorenstein smoothings of $X$, the blowing-down deformation gives an isomorphism between general fibers. 

\begin{lemma}
Let $(W \subset \mathcal W) \to (0 \in \D)$ be a $\Q$-Gorenstein smoothing. Let $P$ be the singularity $\frac{1}{4s}(1,2s-1)$ in $W$. Let $M_P$ be the Milnor fiber corresponding to the induced local smoothing of $P$. Then, there is an exact sequence $$ 0 \to H_2(M_P) \simeq \Z^{s-1} \to H_2(W_t) \to H_2(W) \to 0,$$ where $W_t$ is a smooth fiber (and so an Enriques surface). 
\label{homo}
\end{lemma}

\begin{proof}
Let $W^o=W \setminus \{P\}$. By \cite[Prop.4.2(3)]{Ko05}, we have $H_1(W^o)$ is torsion, so $H^1(W^o)=0$. By \cite[Prop.4.2(1)]{Ko05}, we have $H^1(W^o)\simeq H_3(W)$ and so $H_3(W)=0$. By \cite[Lemma 2.2.3]{H16}, we have then the exact sequence $$ 0 \to H_2(M_P) \to H_2(W_t) \to H_2(W) \to H_1(M_P) \to H_1(W_t) \to 0,$$ since $H_1(W)=0$.
On the other hand, we know that $H_2(M_P)=\Z^{s-1}$, $H_1(M_P)=\Z/2$, and $H_1(W_t)=\Z/2$, and so we have the wanted exact sequence. We recall that locally at $P$, a $\Q$-Gorenstein smoothing of $\frac{1}{4s}(1,2s-1)$ is a $\Z/2$-quotient of a smoothing of the canonical cover $A_{2s-1} \to \frac{1}{4s}(1,2s-1)$, and on the smooth fibers is a topological covering (see e.g. \cite[Section 2.2]{HP10}). 
\end{proof}

\begin{theorem}
Consider the partial resolution $\sigma \colon X \to W$ of $P$ (hence $X$ has $s$ singularities of type $\frac{1}{4}(1,1)$ in a chain of $\P^1$'s), and let us consider an Enriques W-surface $X$. Let $(W \subset \mathcal W) \to (0 \in \D)$ be the induced blowing-down $\Q$-Gorenstein deformation, and assume that it is isomorphism on general fibers (i.e. $X_t \simeq W_t$). 

Then, there is an exact sequence  $$ 0 \to H_2(X_t) \to H_2(X) \to (\Z/2)^{s-1} \to 0.$$ The group $(\Z/2)^{s-1}$ has as generators the images of the $(-1)$-curves between the $\frac{1}{4}(1,1)$ singularities. We have $H_2(X_t)=\Pic(X_t)$, and $H_2(W)=\Cl(W)$ (the class group of $W$).
\label{sequence}
\end{theorem}

\begin{proof}
Let $\Gamma_i$ be the images in $X$ of the $(-1)$-curves between the $C_i$. This short exact sequence is achieved through Lemma \ref{homo}, starting with the restriction morphism $$H_2(X_t) \to H_2(X) \simeq H_2(W) \oplus_{i=1}^{s-1} \Z \Gamma_i.$$ For more details see \cite[Section 8]{TU21}, where we work out a more general picture for $\Q$-Gorenstein smoothings. For the last part see \cite[p.1196]{H13}.
\end{proof}

\begin{remark}
Since $q(X)=p_g(X)=0$, we have identifications $\Pic(X_t)=H^2(X_t)=H_2(X_t)$, $\Pic(X)=H^2(X) \subset H_2(X)=\Cl(X)$ (see \cite{Ko05}), and Mumford's intersection theory on $X$ which respects the topological intersection theory when restricted to $H^2(X)$. Hence the restriction map $H_2(X_t) \to H_2(X)$ of the sequence in Theorem \ref{sequence} respects intersections. 
\end{remark}

\begin{remark}
Let $\Gamma$ be the image in $X$ of a $(-1)$-curve from the chain  $$C_1 - (-1) - C_2 - (-1) - \ldots - (-1) - C_s.$$ Any divisor $D$ in $\Cl(X)$ corresponds to $D=\sum_{i=1}^{s-1} d_i \Gamma_i$ in $(\Z/2)^{s-1}$, where $d_i$ is the local multiplicity of $D$ at any of the singularities in $\Gamma_i$. Note that numerically $$\phi^*(D) \equiv D' + \sum_{i=1}^{s} \mu_i C_i,$$ where $D'$ is the strict transform, and so $D' \cdot C_i=4 \mu_i= d_i$ for $i=1,\ldots, s-1$ as $H_1=\Z/4$ for the link at each singularity. 
\label{localmult}
\end{remark}

\section{Coble-Mukai lattice and Enriques lattice} \label{s3}

Let $X$ be an Enriques W-surface as in Theorem \ref{sequence}, and let $V$ be the associated Coble surface of K3 type. Let $\sigma \colon V \to X$ be the contraction of the boundary $\{C_1,\ldots,C_s\}$.  We have the short exact sequence  $$ 0 \to \Pic(X_t) \to \Cl(X) \to (\Z/2)^{s-1} \to 0.$$ (Note that for $s=1$ it also works, and so $\Pic(X_t) \simeq \Cl(X)$ in that case.)

\begin{lemma}
The image of $\Pic(X_t)$ in $\Cl(X)$ is the set of classes whose proper transform have even intersection number with each $C_i$.
\label{int}
\end{lemma}

\begin{proof}
In Remark \ref{localmult} we have $$\phi^*(D) \equiv D' + \sum_{i=1}^{s} \mu_i C_i,$$ where $D' \cdot C_i=4 \mu_i=d_i$ and $(d_1,\ldots d_{s-1})$ is the image of $D$ in $(\Z/2)^{s-1}$. But $\Pic(X_t)$ is exactly the kernel. 
\end{proof}

\begin{theorem}
The image of $\Pic(X_t)$ in $\Cl(X)$ quotient by $\langle \O_{X_t}(K_{X_t}) \rangle$ is isomorphic to $\CM(V)$.
\label{cm=pic}
\end{theorem}

\begin{proof}
We prove this through the pull-back morphism $$\phi^* \colon \Cl(X) \to \widetilde{\Pic}(V).$$ First, the pull-back on any class in $\Cl(X)$ is orthogonal to all $\beta_i$. So we now restrict to $\phi^* \colon \Pic(X_t) \to \CM(V)$. 

Let $D= D'+ \sum_{i=1}^s \frac{a_i}{2} \beta_i$ in $\CM(V)$ with $D'$ not supported at the $\beta_i$, and $a_i \in \Z$. Then by definition $D \cdot \beta_i=0$ and so $D' \cdot C_i$ is even for all $i$, and so $\phi^*$ is onto, by Lemma \ref{int}. If $\phi^*(D) = D'+ \sum_{i=1}^s \frac{a_i}{2} \beta_i=0$, then $\phi^*(2D)=0$ in $\Pic(V)$. Say that $D \neq 0$ in $\Pic(X_t)$, and so we have a numerical 2-torsion class, and this implies $D \sim K_{X_t}$ by Riemann-Roch on $D$.    
\end{proof}

In \cite[Theorem 9.2.15]{Enr20II}, it is proved that the Coble-Mukai lattice of a Coble surface is isomorphic to the Enriques lattice over $\C$ by a different method. 

\section{Explicit computations for the first degeneration} \label{s4}

In this section, we explicitly compute the Picard group of an Enriques surface from the class group of $X$, which has two $\frac{1}{4}(1,1)$ singularities and it is constructed as follows. Consider a pencil of cubic curves generated by two nodal cubics $C_1, C_2$ which intersect at nine distinct points. Let $S \to \P^2$ be the blow-up at those points, and so $S$ has an elliptic fibration with two $I_1$ fibers $C_1$, $C_2$ (proper transforms of $C_1, C_2$), and at least nine sections $R_1,\ldots,R_9$ from the nine points blown-up. Let $\pi \colon V \to S$ be the blow-up at the nodes of $C_1$ and $C_2$, and let $E_1, E_2$ be the exceptional curves. Let $H$ be the pull-back in $V$ of a general line in $\P^2$.

As before, we contract the $(-4)$-curves $C_1$ and $C_2$ (proper transforms of $C_1$ and $C_2$) to obtain $X$ which has no local-to-global obstructions to deform. Let $H$, $E_1$, $E_2$, $R_1$, $\ldots$, $R_9$ be the images of the corresponding curves in $X$. We see them as classes in $\Cl(X)$. In this way, we can easily compute $$\Cl(X)= \frac{\langle H,E_1,E_2,R_1,\ldots,R_9 \rangle}{ \langle 2E_1- 2E_2,2E_1-3H+R_1+R_2+\ldots+R_9 \rangle}.$$ Indeed, in $\Cl(V)$ we have that $C_i = 3H-2E_i-R_1-R_2-\ldots-R_9$, $C_1-C_2=-2E_1+2E_2$, and $C_1,C_2$ are contracted by $V \to X$.  

We have $H^2=11/2$, $H\cdot R_i=3/2$, $R_i^2=-1/2$, $E_i^2=0$, and $R_i\cdot R_j=1/2$ for $i \neq j$. Consider an Enriques W-surface $X$ as in Theorem \ref{sequence}, so that we have the short exact sequence 

\begin{equation} 
0 \to \Pic(X_t) \to \Cl(X) \to \Z/2 \to 0. 
\label{seq} 
\end{equation} 
\vspace{0.05cm}

We can think of $\Z/2$ as generated by the image (say) of $R_9$ (any $R_i$ works in this case, since they are locally toric boundaries for both singularities). Note that $$K_{X} = E_1-E_2,$$ and it represents the canonical class of $X_t$.

Using that $\Pic(W_t)$ is the kernel in the short exact sequence (\ref{seq}) and numerical independence of classes, we can compute $$ \Pic(X_t)= \langle E_1,E_1+2R_9,R_1-R_2, R_5-R_6, R_6-R_7, \ \ \ \ \ \ \ \ \ \ \ \ \ \ \ \ \ \ \ \ \ \ \ \ \ \ \ \ \ \ \ \ \ \ \ \ $$ $$ \ \ \ \ \ \ \ \ \ \ \ \ \ \ \ \ \ \ \ \ \ \ \ \ R_2-R_3, R_3-R_4, R_4-R_5,  R_7-R_8,H-R_1-R_2-R_3 \rangle +  K_W,$$ and these generators give a basis for the Enriques lattice $\H \oplus E_8(-1)$, where $$\H=\langle E_1,E_1+2R_9 \rangle$$ and $$E_8(-1)= \langle R_1-R_2,R_2-R_3, R_3-R_4, \ \ \ \ \ \ \ \ \ \ \ \ \ \ \ \ \ \ \ \ \ \ \ \ \ \ \ \ \ \ \ \ \ \ \ \ \ \ \ \ \ \ \ \ \ \ \ \ \ \ \ \ \ \ \ \ \ \ \ \ \ \ \ \ \ \ \ \ \ \ \ \ $$ $$ \ \ \ \ \ \ \ \ \ \ \ \ \ \ \ \ \ \ \ \ \ \ R_4-R_5, R_5-R_6, R_6-R_7, R_7-R_8,H-R_1-R_2-R_3 \rangle.$$

Following the recipe and notation in \cite[Section 1.5]{Enr20I}, we have the root basis for the Enriques lattice $$\alpha_0:=H-R_1-R_2-R_3, \ \ \alpha_1:=R_1-R_2,$$ $$ \ \alpha_2:=R_2-R_3, \ \ \alpha_3:=R_3-R_4,$$ $$\alpha_4:=R_4-R_5, \ \ \alpha_5:=R_5-R_6, $$ $$ \ \ \alpha_6:=R_6-R_7, \ \ \alpha_7:=R_7-R_8,$$ $$ \alpha_8:=-3H+E_1+\sum_{j=1}^7R_j+2R_8, \ \ \ \alpha_9:= 2R_9.$$
 
We also write down the following isotropic sequence (\cite[Section 1.5]{Enr20I}): $$ f_{10}= E_1, \ \ \ f_{9}= E_1 +2 R_9, \ \ \ f_{i}=  -3H + 2 E_1 + \sum_{j=1}^8 E_j + R_i + 2R_9 $$ for $i=1,\ldots,7$, and so $\Delta= \frac{1}{3} (f_1+\ldots+f_{10}) = -8H +6 E_1 + 3 \sum_{j=1}^9 R_j$.

\section{Further degenerations and $\Q$-Homology projective planes} \label{s5}


As explained in \cite[5.9]{Enr20I}, the Baily-Borel compactification of the moduli space of Enriques surfaces is formed by a Coble divisor, which is the moduli space of Coble surfaces, and two smooth rational curves \cite[Theorem 5.9.8]{Enr20I}. Hence the Baily-Borel compactification suggests that we should see a big part ``closer to the boundary" of the moduli space of Enriques surfaces via Enriques W-surfaces (see \cite[5.10]{Enr20I}). In this section we show what sort of things one can actually see. 

We start with the most degenerate Coble surface (see \cite[Example 9.2.7]{Enr20II}). It has $10$ $\frac{1}{4}(1,1)$ singularities, and it is unique up to isomorphism (see \cite[Chapter 9]{Enr20II}). This surface corresponds to one of the two most algebraic K3 surfaces \cite{V83} under the canonical double cover. (In fact, the K3 surface is $X_4$ in Vinberg's notation, which is birational to the $4$-th cyclic cover of $\P^2$ branched at $xyz(x+y+z)=0$.) We take as a model the following construction. We start with the rational elliptic fibration $S' \to \P^1$ which has exactly $4$ singular fibers: one $I_9$ and three $I_1$. This surface has Mordell-Weil group isomorphic to $\Z/3$ (see e.g. \cite{P90}). Figure \ref{f1} gives the notation for all relevant curves.

\begin{figure}[htbp]
\includegraphics[width=3.5in]{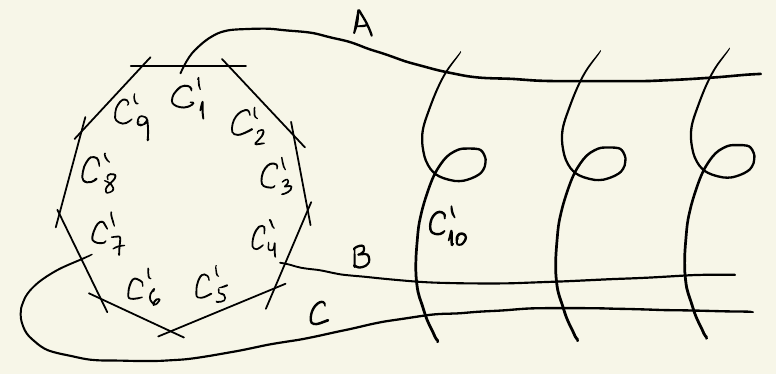}
\caption{Relevant curves in $S'$}
\label{f1}
\end{figure}

Let $\pi' \colon V' \to S'$ be the blow-up at all $9$ nodes in $I_9$, and at the node in $I_1=C'_{10}$. As we have done it before, we keep the notation for proper transforms of curves. Let $E'_i$ be the $(-1)$-curves from the blow-up. Hence we consider in $V'$ the chain $$C'_{10} - A - C'_1 - E'_1 - C'_2 - E'_2 - \ldots - E'_8 - C'_9 $$ where ${C'_i}^2=-4$ and $A^2=-1$. Hense $E'_9$ is between $C'_1$ and $C'_9$, and $E'_{10}$ is intersecting $C'_{10}$ at two distinct points. We now contract all the $C'_i$ via $\phi' \colon V' \to X'$, so $X'$ has $10$ $\frac{1}{4}(1,1)$ singularities, and $K_{X'} \equiv 0$. Indeed, we have that $$C'_{10}+2 E'_{10} \sim F \sim \sum_{i=1}^{9} C'_i + 2 \sum_{i=1}^{9} E'_i,$$ where $F$ is a general fiber of the elliptic fibration $V' \to \P^1$, and $$\phi'^*(K_{X'}) - \frac{1}{2} \sum_{i=1}^{10} C'_{i} \equiv K_{V'} \sim -F + \sum_{i=1}^{10} E'_{i}.$$ The surface $X'$ has no local-to-global obstructions to deform, because we are contracting curves over two fibers of $S' \to \P^1$ (see e.g. \cite[Section 4]{PSU13}). Our notation for relevant curves in $X'$ is $\Gamma_0:= E'_1$, $\Gamma_1:= E'_2$, $\Gamma_2:= E'_3$, $\Gamma_3:= E'_4$, $\Gamma_4:= E'_5$, $\Gamma_5:= E'_6$, $\Gamma_6:= E'_7$, $\Gamma_7:= E'_8$, $\Gamma_8:= E'_9$, $\Gamma_9:= A$, $\Gamma_{10}:= B$, $\Gamma_{11}:=C$, and $\Gamma_{12}:= E'_{10}$, which is represented in Figure \ref{f2}. 

\begin{figure}[htbp]
\includegraphics[width=3.5in]{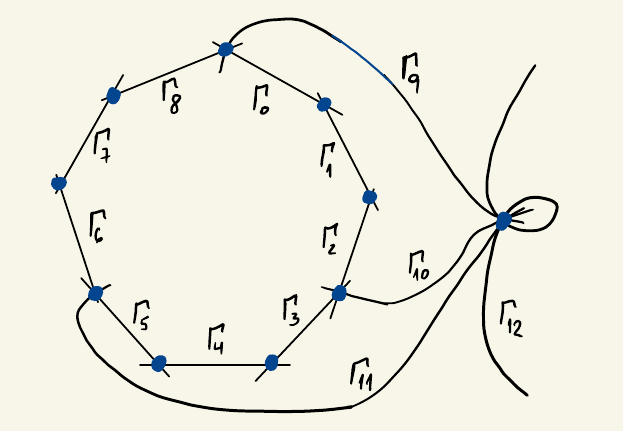}
\caption{Relevant curves in $X'$}
\label{f2}
\end{figure}

In this case we have that $$\Cl(X')=H_2(X')=\langle \Gamma_0,\Gamma_1,\ldots,\Gamma_{12} \rangle/\langle \cR_1, \cR_2, \cR_3 \rangle $$ where the relations are:

\bigskip

$\cR_1:=-2\Gamma_0-2\Gamma_1-2 \Gamma_2 -\Gamma_3+ \Gamma_4+ 3\Gamma_5+ 3\Gamma_6 +\Gamma_7-\Gamma_8 -\Gamma_9 -\Gamma_{10} +2\Gamma_{11} $

\bigskip

$\cR_2:=-\Gamma_0 -3\Gamma_1 -5\Gamma_2 -4\Gamma_3 + 4\Gamma_5+ 5\Gamma_6 +3\Gamma_7 +\Gamma_8 -3\Gamma_{10} + 3 \Gamma_{11}$

\bigskip

$\cR_3:=-\Gamma_0 -3\Gamma_1 -5\Gamma_2 -3\Gamma_3+ 3\Gamma_4+ 9\Gamma_5+ 10\Gamma_6+ 6\Gamma_7+ 2\Gamma_8 -\Gamma_9 -4\Gamma_{10}+ 5\Gamma_{11} -2 \Gamma_{12}$

\bigskip

Let $X$ be the surface constructed in Section \ref{s4}. We have the curves $H,E_1,E_2,R_1,\dots,R_9$. One can see $X$ as a $\Q$-Gorenstein smoothing of $X'$. Indeed we could choose a general $\Q$-Gorenstein smoothing of $X'$ which smooths all singularities except the singularities at $\Gamma_0 \cap \Gamma_8$ and at the node of $\Gamma_{12}$. We can recover all curves $H,E_1,E_2,R_1,\dots,$ $R_9$ from linear combinations of the curves $\Gamma_0$, $\ldots$, $\Gamma_{12}$, and in particular we can read $\Pic(X_t)$ from $\Cl(X')$. Details in a more general setup will be given in \cite{TU21}. We note that in this case $$ 0 \to \Pic(X_t) \to \Cl(X') \to (\Z/2)^{9} \to 0.$$

On the other hand, we note that we have many subchains in $\Gamma_0,$ $\dots,$ $\Gamma_{11}$ which can be contracted to singularities of type $[3,2,\dots,2,3]$.  The relevance of that is the application of Proposition \ref{ade} to particular deformations of $W'$, which is the contraction of some disjoint chains of $\Gamma_i$'s. Let us take the maximal chain $\Gamma_1+\dots+\Gamma_9$. Let $X' \to W'$ be the contraction of that chain. In this way, the surface $W'$ has one T-singularity, whose minimal resolution corresponds to the Hirzebruch-Jung continued fraction $$\frac{40}{19}=[3,2,2,2,2,2,2,2,2,3].$$ We recall that $W'$ has no local-to-global obstructions to deform (same reason as for $X'$), and so we can choose from Proposition \ref{ade} a global $\Q$-Gorenstein deformation which deforms $\frac{1}{40}(1,19)$ into an $A_9$ rational double point. This is a one dimensional family. In this way, the general fiber $W'_t$ is a Gorenstein $\Q$-Homology projective plane with $K\equiv 0$. We recall that a normal projective surface with only quotient singularities whose second Betti number is $1$ is called a  $\Q$-Homology projective plane, and it is Gorenstein if all singularities are rational double points. In \cite{HKO15}, Hwang, Keum and Ohashi classified all possible configurations of singularities for Gorenstein $\Q$-Homology projective planes. In the case of $K \equiv 0$ there are $31$ possible configurations, and they were able to produce examples for $29$ of them. We can have at most $5$ singularities in these configurations. Of course, the minimal resolution of all of them are Enriques surfaces with particular ADE configurations of nine $(-2)$-curves. 

In \cite{S21}, Sch\"utt classifies all Gorenstein $\Q$-Homology projective planes. It turns out that for each of the $31$ types we have a one dimensional moduli space, and the number of components for each type varies from $1$ to $3$:

\begin{itemize}
\item $3$ components for $A_7+2A_1$, and $3A_3$;

\item $2$ components for $A_8+A_1$, $A_7+A_2$, $A_5+A_3+A_1$, $A_5+2A_2$, $D_8+A_1$, $D_6+A_3$, $D_4+A_3+2A_1$, $E_8+A_1$, $E_7+A_2$, $E_6+A_3$, $E_7+A_2$, $E_8+A_1$, $D_9$, and $D_7+2A_1$;

\item all the other $15$ root types.
\end{itemize} 

\begin{theorem}
The one dimensional family from $W'$ whose general fiber has an $A_9$ singularity is in closure of the Sch\"utt irreducible moduli corresponding to $A_9$.
\end{theorem}

We can actually obtain almost all of the $31$ types from Enriques W-surfaces just as explained for $A_9$, and a small variant when we meet the type $D$ or $E$ singularities. We only give two examples below, but it is possible to realize at least $27$ types. 

\begin{example} Consider the rational elliptic fibration with sections $S \to \P^1$ whose singular fibers are $I_4^*+2I_1$ (see e.g. \cite[p.7]{P90}). It has exactly two sections. Let us fix one of these sections, we denote it by $A$. Let $B$ and $C$ be the $I_1$ fibers, and let $D$ be $I_4^*$ minus the component intersecting $A$. We blow-up the node at $B$ and the node at $C$, and then we contract the section $A$. We obtain a surface $V$ with $B^2=C^2=-3$, and $D$ is a $D_8$ configuration, i.e., the minimal resolution of a $D_8$ singularity. The contraction of $B+C$ produces a surface $W$ with a singularity $[3,3]$, and it has no-global-to-local obstructions. Moreover, since $D$ is disjoint to $B+C$, the contraction $W'$ of $D$ in $W$ has also no local-to-global obstructions \cite[Section 4]{PSU13}. Therefore, by Proposition \ref{ade}, we can choose to $\Q$-Gorenstein deform $W'$ so that we keep the $D_8$ singularity and we deform $[3,3]$ into $A_1$. In this way, we obtain a $\Q$-Homology projective plane with $K \equiv 0$ and root type $D_8+A_1$. According to Sch\"utt's result, this root type corresponds to a moduli curve with two components. 

Similarly, we can construct a surface for the root type $D_5+A_4$ from the rational elliptic surface with sections whose singular fibers are $I_1^*+I_4+I_1$. Here the Mordell-Weil group is $\Z/4$ \cite{P90}. Just choose one of the four sections, and blow-up a suitable node in $I_4$ and the node in $I_1$, and then consider $D_5$ and $[3,2,2,2,3]$ to construct what we want through Proposition \ref{ade}. This root type has an irreducible moduli. To get other root types, we need to consider more involved situations. For example, we may start with a special Halphen surface of index $2$ and some suitable double sections.        
\label{d}
\end{example}

Compactifications of Sch\"utt moduli curves should correspond to degenerations of Enriques surfaces, since there are no non-isotrivial smooth families of Enriques surfaces over proper bases by \cite[Corollary 5.9.13]{Enr20I}. Hence, a natural question is: \textit{Find compactifications of Sch\"utt moduli curves and the corresponding limit singular surfaces. Which root types hit Enriques W-surfaces? What else?}


\end{document}